\documentclass{amsart}

\usepackage{amsmath,amsfonts,amsthm,amssymb,stmaryrd,paralist,tikz,amsthm}
\usepackage[mathscr]{euscript}
\usetikzlibrary{matrix,arrows,decorations.pathmorphing}

\newcommand{\R}{\mathbb{R}}

\newcommand{\C}{\mathbb{C}}

\newcommand{\CP}{\mathbb{C}\mathrm{P}}
\renewcommand{\H}{\mathbb{H}}

\newtheorem{thm}{Theorem}[section]
\newtheorem*{thm*}{Theorem}
\newtheorem{lem}[thm]{Lemma}
\newtheorem*{lem*}{Lemma}
\newtheorem{cor}[thm]{Corollary}
\newtheorem*{cor*}{Corollary}
\newtheorem{prop}[thm]{Proposition}
\newtheorem*{prop*}{Proposition}
\newtheorem{defn}{Definition}
\newtheorem*{defn*}{Definition}

\newtheorem*{question*}{Question}

\usepackage{color}
\definecolor{verydarkblue}{rgb}{0,0,0.4}
\usepackage{hyperref}
\hypersetup{
pdfauthor={Keaton Quinn},
pdftitle={Asymptotically Poincar\'e surfaces in quasi-Fuchsian manifolds},
colorlinks=true,linkcolor=verydarkblue,
citecolor=verydarkblue,urlcolor=verydarkblue
}

\begin{document}

\title{Asymptotically Poincar\'e surfaces in quasi-Fuchsian manifolds}

\author{Keaton Quinn}
\address{Department of Mathematics, Statistics and Computer Science, University of Illinois at Chicago, Chicago, IL}
\email{\href{mailto:kquinn23@uic.edu}{kquinn23@uic.edu}}

%\subjclass[2010]{Primary: 30F60 (Teichm\"uller Theory), Secondary: 53C42 (Immersion (minimal, prescribed curvature, tight, etc.)).}

\date{July 31, 2019 (v1: November 21, 2018)}

\begin{abstract}
We introduce the notion of an asymptotically Poincar\'e family of surfaces in an end of a quasi-Fuchsian manifold.
We show that any such family gives a foliation of an end by asymptotically parallel convex surfaces, and that the asymptotic behavior of the first and second fundamental forms determines the projective structure at infinity.
As an application, we establish a conjecture of Labourie from \cite{labourie1992} regarding constant Gaussian curvature surfaces. We also derive consequences for constant mean curvature surfaces.
\end{abstract}

\maketitle

\section{Introduction}

Epstein in \cite{epstein1984} describes a map that takes a domain $\Omega$ in $\CP^1$ and a conformal metric $\sigma$ on $\Omega$ and gives an \textit{Epstein surface} $\mathrm{Ep}_\sigma : \Omega \to \H^3$ in hyperbolic space.
A variant of this construction associates to any conformal metric $\sigma$ on the Riemann surface at infinity $X$ of an end of a quasi-Fuchsian manifold $M$, a map $X \to M$, which we also call an Epstein surface. 
We use this construction to study surfaces in quasi-Fuchsian manifolds. 

The hyperbolic metric $h$ of $X$ gives a distinguished Epstein surface that we call the \textit{Poincar\'e surface of $X$}. 
Scalar multiples of $h$ give a family of Epstein surfaces, which we call \text{the Poincar\'e family}, that foliate the end. 
Similarly, a family of conformal metrics that is $C^\infty$-asymptotic to the family of scalar multiples of $h$ gives rise to a family of Epstein surfaces that we call an \textit{asymptotically Poincar\'e family}.
These definitions are made precise in Section \ref{asym-def-section} (see Definition \ref{asym-def}).

Our main result concerns the first and second fundamental forms, $I_\epsilon$ and $I\!I_\epsilon$, of an asymptotically Poincar\'e family $S_\epsilon$, for $\epsilon \in (0,1)$. These both give families of Riemannian metrics on $X$, and hence, points in Teichm\"uller space $[I_\epsilon]$ and $[I\!I_\epsilon]$. We prove the following regarding their asymptotic behavior.

\begin{thm}
\label{big-thm-intro}
In a quasi-Fuchsian manifold, let $S_\epsilon$, for $\epsilon \in (0,1)$, be an asymptotically Poincar\'e family of surfaces with surface at infinity $X$.
If $I_\epsilon$ and $I\!I_\epsilon$ are the corresponding first and second fundamental forms then 
\[
[I_\epsilon] \to [h] 
\quad \text{ and } \quad 
[I\!I_\epsilon] \to [h]
\quad \text{ as } \epsilon \to 0
\] 
and the tangent vector to $[I\!I_\epsilon]$ at $[h]$ vanishes while the tangent vector to $[I_\epsilon]$ is determined by the holomorphic quadratic differential $\phi$ that measures the difference between the induced projective structure and the Fuchsian projective structure on $X$.
That is,
\[
\dot{[I_\epsilon]} = c \, \mathrm{Re}(\phi) \quad \text{and } \quad \dot{[I\!I_\epsilon]} = 0
\]
for a positive constant $c$.
\end{thm}
Following \cite{schlenker2017}, we call $\phi$ the holomorphic quadratic differential at infinity of $M$.
Also, $[I_\epsilon]$ refers to the conformal class of the Riemannian metric $I_\epsilon$ and $\dot{[I_\epsilon]}$ denotes $\left. \frac{\partial}{\partial \epsilon} [I_\epsilon] \right|_{\epsilon = 0}$.
However, a comment is needed regarding $[I\!I_\epsilon]$.
Due to the convenient choice of normal vector as described in Section \ref{asym-def-section}, the second fundamental form is actually negative definite. 
Its negative, therefore, defines a point in Teichm\"uller space and so we will use $[I\!I_\epsilon]$ to denote the conformal class of $-I\!I_\epsilon$.

As an application of our results we answer a conjecture posed by Labourie.
In \cite{labourie1991} he proves that an end of a quasi-Fuchsian manifold admits a unique foliation by $k$-surfaces: surfaces $S_k$, for $k$ in $(-1,0)$, such that the Gaussian curvature of $S_k$ is identically $k$. 
In \cite{labourie1992}, Labourie notes how $k$-surfaces may be interpreted as a path in Teichm\"uller space and he asks what the tangent vectors to the paths $[I_k]$ and $[I\!I_k]$ are at $k=0$. 
He conjectures that they are related to the holomorphic quadratic differential at infinity.
We show in Section \ref{k-surfaces-section} that these $S_k$ form an asymptotically Poincar\'e family of surfaces.
Theorem \ref{big-thm-intro} then proves his conjecture and gives the relationship explicitly.

\begin{thm}
\label{k-surfaces-intro}
Let $I_k$ and $I\!I_k$ be the first and second fundamental forms of the $k$-surface $S_k$. 
Let $\phi$ be the holomorphic quadratic differential at infinity of $M$. 
Then, as $k \to 0$, the tangent vectors to $[I_k]$ and $[I\!I_k]$ in the Teichm\"uller space of $X$ are given by 
\[
  \dot{[I_k]}= - \mathrm{Re}(\phi) \quad \text{and } \quad \dot{[I\!I_k]} = 0.
\]
\end{thm}

We consider another application of Theorem \ref{big-thm-intro} in Section \ref{mean-curvature-section}.
The work of Mazzeo and Pacard in \cite{mazzeo-pacard2011} show that the ends of a quasi-Fuchsian manifold admits a unique foliation by surfaces of constant mean curvature. 
We prove that this family of surfaces forms an asymptotically Poincar\'e family of surfaces by constructing, for each negative $k$ near zero, an Epstein surface whose mean curvature is identically $-\sqrt{1+k}$.
Since asymptotically Poincar\'e surfaces foliate an end of $M$, by the uniqueness result of \cite{mazzeo-pacard2011} these surfaces are those shown to exist by Mazzeo and Pacard.
Therefore, Theorem \ref{big-thm-intro} describes the behavior of this family in Teichm\"uller space.  

\begin{thm}
\label{cmc-intro}
Let $I_k$ and $I\!I_k$ be the first and second fundamental forms of the Epstein surface with constant mean curvature $-\sqrt{1+k}$.
Let $\phi$ be the holomorphic quadratic differential at infinity of $M$. 
Then, as $k \to 0$, the tangent vectors to $[I_k]$ and $[I\!I_k]$ in Teichm\"uller space are given by 
\[
  \dot{[I_k]}= - \mathrm{Re}(\phi) \quad \text{and } \quad \dot{[I\!I_k]} = 0.
\]
\end{thm}

\subsection*{Acknowledgments}
I would like to thank my Ph.D. thesis advisor David Dumas for suggesting this problem and for the many helpful conversations that followed.
The author was partially supported in Summer 2018 by a research assistantship under NSF DMS-1246844, RTG: Algebraic and Arithmetic Geometry, at the University of Illinois at Chicago.

%***********************************************
\section{Preliminaries}
%***********************************************

\subsection{Quasi-Fuchsian Manifolds}

For a discrete subgroup $\Gamma$ of $\mathrm{SL}(2,\C)$, let $\Lambda(\Gamma)$ denote its limit set and $\Omega(\Gamma) = \CP^1 - \Lambda(\Gamma)$ the domain of discontinuity. 
If the limit set $\Lambda(\Gamma)$ is a Jordan curve then the domain of discontinuity is separated into two domains $\Omega(\Gamma) = \Omega_+ \cup \Omega_-$. 
When $\Omega_+$ and $\Omega_-$ are each $\Gamma$-invariant, the subgroup $\Gamma$ is called quasi-Fuchsian and the manifold $M = \H^3/\Gamma$ is a quasi-Fuchsian manifold.
The convex hull of $\Lambda(\Gamma)$ in hyperbolic space is $\Gamma$-invariant and its quotient in $M$ is called the convex core.
The complement of the convex core in $M$ consists of two ends diffeomorphic to $\Omega_{\pm}/\Gamma \times (0,\infty)$, and the surfaces $\Omega_{\pm}/\Gamma$ are called the surfaces at infinity for their respective ends. 
We will be focusing on one end of a quasi-Fuchsian manifold with smooth surface at infinity $\Omega/\Gamma$.

\subsection{Projective Structures and the Schwarzian derivative}
A complex projective structure on a surface $S$ is an atlas of charts to $\CP^1$ whose transition functions are restrictions of M\"obius transformations. 
The surface at infinity $\Omega/\Gamma$ of an end of $M$ is both a Riemann surface $X$ and a complex projective surface; we denote the latter by $X_M$ to emphasize its dependence on the manifold $M$. 
The Riemann surface $X$ also has its standard Fuchsian projective structure $X_F$ coming from the uniformization theorem. 
The difference between these two projective structures is the quadratic differential $\phi = X_M - X_F$, which is holomorphic with respect to $X$ (see e.g., \cite{dumas2009}).
We call this the holomorphic quadratic differential at infinity of $M$.

Here is a more concrete description of $\phi$.
Recall that if $f$ is a locally injective holomorphic function on a domain in $\C$ or $\CP^1$ then the Schwarzian derivative $\mathcal{S}(f)$ of $f$ is a holomorphic quadratic differential.
Then $\tilde{\phi} = \mathcal{S}(f)$, where $f: \Omega \to \H$ is a Riemann map, is a $\Gamma$-invariant tensor that induces the holomorphic quadratic differential at infinity $\phi$.

%*************************
\section{Epstein Surfaces}
%*************************

\subsection{Conformal Metrics} A Riemann surface structure on a surface $S$ distinguishes a conformal class of metrics on $S$, called conformal metrics. 
These are metrics $\sigma$ that, in a complex coordinate chart $z$, can be written as $\sigma = e^{2\eta} |dz|^2$, for some real valued function $\eta$. 
This $\eta$ is called the log density function of $\sigma$ and we say the metric $\sigma$ is of class $C^k$ if $\eta$ is a $C^k$ function. 
For a $C^2$ conformal metric $\sigma = e^{2\eta}|dz|^2$, the Gaussian curvature is the function $K(\sigma) = -4 e^{-2\eta}\eta_{z\bar{z}}$. 
For a compact Riemann surface with genus larger than 1, this conformal class contains a unique hyperbolic metric  which we denote by $h$. 
If we call $X$ the Riemann surface structure on $S$ then we will denote the set of all smooth conformal metrics on $S$ by $\mathrm{Conf}^\infty(X)$.

\subsection{The Visual Metric}
A natural trivialization of the unit tangent bundle of hyperbolic space is $U \H^3 \cong \H^3 \times \CP^1$ given by sending a tangent vector $v$ at a point $p$ to the ideal endpoint of the geodesic through $p$ in the direction $v$.
For each $p$ we may use this trivialization to push forward the induced metric on $U_p\H^3$ to $\CP^1$, obtaining $V_p$, a conformal metric on $\CP^1$ called the visual metric from $p$. 
As an example, the visual metric from the origin in the ball model $V_0$ is just the spherical metric on $S^2$, which is identified with $\CP^1$ in this model. 
In general, if $M$ is a M\"obius transformation taking $0$ to the point $p$, then $V_p = M_*V_0$.

\subsection{The Epstein Map}
In \cite{epstein1984} Epstein gives a way of constructing surfaces in hyperbolic space from domains in the Riemann sphere. 
These Epstein surfaces will be our main tool for studying asymptotically Poincar\'e surfaces. 
Our exposition of this Epstein construction follows \cite{dumas2017} (see also \cite{anderson1998}).

\begin{thm}[Epstein \cite{epstein1984}]
Let $\Omega$ be a domain in $\CP^1$  and $\sigma$ a $C^k$ conformal metric on $\Omega$, then there exists a unique $C^{k-1}$ map $\mathrm{Ep}_\sigma : \Omega \to \H^3$, called the Epstein map of $\Omega$ for the metric $\sigma$, such that for all $z \in \Omega$,
\[
V_{\mathrm{Ep}_\sigma(z)}(z) = \sigma(z).
\]
Moreover, the image of a point $z$ depends only on the 1-jet of $\sigma$ at $z$.
\end{thm}

Epstein's original construction gave a formula for $\mathrm{Ep}$ in the ball model of hyperbolic space. Dumas gives an $\text{SL}(2,\C)$-frame field description of the map as follows.
Choose an affine chart $z$ on $\CP^1$ that distinguishes a point $0 \in \Omega$ and $\infty \notin \Omega$. Then, on the geodesic in $\H^3$ with ideal endpoints $0$ and $\infty$, there exists a unique point $p$ such that the visual metric from $p$ at $0$ is the Euclidean metric of this affine chart, $V_p(0) = |dz|^2$. The Epstein map is an $\mathrm{SL}(2,\C)$-frame orbit of this point.

\begin{prop}[\cite{dumas2017}]
\label{dumas-def}
On a domain $\Omega$ in $\CP^1$ write $\sigma = e^{2\eta}|dz|^2$. Define the $\mathrm{SL}(2,\C)$-frame field $\widetilde{\mathrm{Ep}}_\sigma: \Omega \to \mathrm{SL}(2,\C)$ by 
\[
\widetilde{\mathrm{Ep}}_\sigma(z) =
\begin{pmatrix}
1 & z \\
0 & 1
\end{pmatrix}
\begin{pmatrix}
1 & 0 \\
\eta_z & 1
\end{pmatrix}
\begin{pmatrix}
e^{-\eta/2} & 0 \\
0 & e^{\eta/2}
\end{pmatrix},
\]
then the Epstein map is given by 
\[
\mathrm{Ep}_\sigma(z) = \widetilde{\mathrm{Ep}}_\sigma(z) \cdot p.
\]
\end{prop}

Even though we call the image an Epstein surface, the Epstein map need not be an immersion. 
Indeed, if $\sigma$ is itself a visual metric then the Epstein map for $\sigma$ is constant. 
However, the lift of $\mathrm{Ep}_\sigma$ from $\Omega$ to the unit tangent bundle of hyperbolic space given by $\widehat{\mathrm{Ep}}_\sigma(z) = (\text{Ep}_\sigma(z), z)$ is an immersion (recall the trivialization $U\H^3 \cong \H^3 \times \CP^1$). 
This lift can be thought of as providing a unit ``normal'' vector field for the Epstein surface even when the Epstein map is not an immersion. Indeed, this lift agrees with a unit normal vector field when the surface is immersed and so we will simply refer to it as the normal field from now on. 

Because the Epstein map is unique, it is natural with respect to the action of $\mathrm{SL}(2,\C)$ in the following sense. 
Suppose $M$ is a M\"obius transformation, then the following diagram commutes:
\[
\begin{tikzpicture}[scale=0.7]
\node (1) at (0,2) {$(\Omega,\sigma)$};
\node (2) at (3,2) {$(M(\Omega),M_*\sigma)$};
\node (3) at (0,0) {$\H^3$};
\node (4) at (3,0) {$\H^3$};

\draw[->] (1) to node [above] {$M$} (2);
\draw[->] (1) to node [left] {$\mathrm{Ep}$} (3);
\draw[->] (3) to node [above] {$M$} (4);
\draw[->] (2) to node [right] {$\mathrm{Ep}$} (4);

\end{tikzpicture}
\]
That is, $\mathrm{Ep}_{M_*\sigma}(M(z)) = M( \mathrm{Ep}_{\sigma}(z))$. See \cite{anderson1998}.

This allows us to define Epstein maps on certain quotients. 
Suppose in general that $\Gamma$ is a subgroup of $\mathrm{SL}(2,\C)$ acting freely and properly discontinuously on $\H^3 \cup \CP^1$ leaving a domain $\Omega$ invariant. 
Then $\Omega/\Gamma$ inherits a Riemann surface structure. 
Call this structure $X$ and let $\sigma$ be a conformal metric  on $X$.
Lift this to $\tilde{\sigma}$ on $\Omega$, which is $\Gamma$-invariant. 
Then $\mathrm{Ep}_{\tilde{\sigma}}: \Omega \to \H^3$ is $\Gamma$-equivariant and therefore descends to a map $\mathrm{Ep}_\sigma : X \to \H^3/ \Gamma$. 
In particular, when $\Gamma$ is a quasi-Fuchsian group and $\Omega$ a component of the domain of discontinuity, each conformal metric $\sigma$ on the surface at infinity $X$ gives rise to a map from $X$ into the quasi-Fuchsian manifold $M$.

\subsection{Parallel Surfaces}

Let $g^t : U \H^3 \to \H^3$ denote the time-$t$ geodesic flow projected down to $\H^3$.
Thus for a unit tangent vector $v$ on $\H^3$ we have $g^t(v) = \exp_p(tv)$.
Using the lift of an Epstein surface to $U\H^3$ described above, each Epstein surface gives rise to a family of surfaces by applying the geodesic flow (and projecting to $\H^3$). 
That is, we have the flowed surfaces $g^t \circ \widehat{\mathrm{Ep}}_\sigma(\Omega)$.
In fact, these surfaces are themselves Epstein surfaces corresponding to scalar multiples of $\sigma$:

\begin{lem}[Thurston, see \cite{epstein1984}]
\label{epstein-flow}
Let $\Omega$ be a domain in $\CP^1$ and $\sigma$ a conformal metric on $\Omega$.
Then 
\[
g^t \circ \widehat{\mathrm{Ep}}_\sigma  = \mathrm{Ep}_{e^{2t} \sigma}.
\]
That is, flowing the Epstein surface for $\sigma$ for time $t$ corresponds to taking the Epstein surface for the metric $e^{2t}\sigma$.
\end{lem}

\subsection{Schwarzian Derivatives of Conformal Metrics}
In general, the Schwarzian derivative of one conformal metric with respect to another on a Riemann surface is the $(2,0)$ part of the Schwarzian tensor defined by Osgood and Stowe in \cite{osgood-stowe1992}. 
For two conformal metrics $\sigma_1$ and $\sigma_2$ and a coordinate chart $z$,  write $\sigma_i = e^{2\eta_i} |dz|^2$,  then the Schwarzian derivative of $\sigma_2$ with respect to $\sigma_1$ is the quadratic differential 
\[
B(\sigma_1,\sigma_2) = \left( (\eta_2)_{zz} - (\eta_2)_z^2 - (\eta_1)_{zz} + (\eta_1)_z^2 \right) dz^2.
\]
As opposed to the Schwarzian derivative of a function, this need not be holomorphic. 
The two are related, though. 
For $f$ locally injective and holomorphic, we have
\[
\mathcal{S}(f) = 2B(|dz|^2,f^*|dz|^2).
\]
We also have naturality $f^*B(\sigma_1,\sigma_2) = B(f^*\sigma_1,f^*\sigma_2)$ (again, for $f$ holomorphic) and a cocycle property $B(\sigma_1,\sigma_3) = B(\sigma_1,\sigma_2) + B(\sigma_2,\sigma_3)$. 
When $\sigma$ is a conformal metric on a domain in $\C$ with $B(|dz|^2,\sigma) = 0$ it is called a M\"obius flat metric.
We will denote any such metric by $g_{\CP^1}$. 
The cocycle property gives us that the Schwarzian derivative of $\sigma = e^{2\eta}|dz|^2$ relative to any M\"obius flat metric is 
\[
B(g_{\CP^1},\sigma) = (\eta_{zz} - \eta_{z}^2 )dz^2.
\]

\subsection{Geometry of Epstein Surfaces}
The first fundamental form of the Epstein surface for the metric $\sigma$ is given by $I(\sigma) = \mathrm{Ep}_\sigma^*(g_{\H^3})$ for $g_{\H^3}$ the metric of $\H^3$. 
It is given by 
\[
I(\sigma) = \frac{4}{\sigma}|B(g_{\CP^1},\sigma)|^2 + \frac{1}{4}(1-K(\sigma))^2\sigma + 2(1-K(\sigma))\text{Re}(B(g_{\CP^1},\sigma)).
\]
The second fundamental form (relative to the normal lift $\widehat{\mathrm{Ep}}_\sigma$) is 
\[
I\!I(\sigma)
= \frac{4}{\sigma}|B(g_{\CP^1},\sigma)|^2 - \frac{1}{4} (1 - K(\sigma)^2)\sigma - 2 K(\sigma) \text{Re}(B(g_{\CP^1},\sigma))
\]
(see \cite[Eqns.~3.2-3.3]{dumas2017}).
Here $K(\sigma)$ is the Gaussian curvature of $\sigma$ and $B(g_{\CP^1},\sigma)$ the Schwarzian derivative of $\sigma$ with respect to a M\"obius flat metric.
With $I(\sigma)$ and $I\!I(\sigma)$ we can compute the Gaussian curvature by $K(I(\sigma)) = -1 + \det(I(\sigma)^{-1}I\!I(\sigma))$ and the mean curvature by $H(\mathrm{Ep}_\sigma) = \frac{1}{2}\mathrm{tr}(I(\sigma)^{-1}I\!I(\sigma))$. 
We obtain
\[
K(I(\sigma))
= \frac{4K(\sigma)}{(1-K(\sigma))^2 - \frac{16}{\sigma^2} |B(g_{\CP^1},\sigma)|^2}
\]
and
\[
H(\mathrm{Ep}_\sigma)
= \frac{K(\sigma)^2 - 1 - \frac{16}{\sigma^2}|B(g_{\CP^1},\sigma)|^2}{(K(\sigma) - 1)^2 - \frac{16}{\sigma^2}|B(g_{\CP^1},\sigma)|^2}.
\]

In the quasi-Fuchsian setting, if $\sigma$ is a $\Gamma$-invariant conformal metric on $\Omega$ then each term in the above equations is also $\Gamma$-invariant. This is maybe less clear for the quadratic differential $B(g_{\CP^1},\sigma)$ since the M\"obius flat metric $g_{\CP^1}$ is not itself $\Gamma$-invariant. However, we see that for $\gamma \in \Gamma$ we have $\gamma^*B(g_{\CP^1},\sigma) = B(\gamma^*g_{\CP^1},\gamma^*\sigma) = B(\gamma^* g_{\CP^1},\sigma)$, by naturality of Schwarzian derivatives of conformal metrics. 
The metric $\gamma^* g_{\CP^1}$ is still a M\"obius flat metric, and so $B(\gamma^* g_{\CP^1},\sigma) = B(g_{\CP^1},\sigma)$, implying $B(g_{\CP^1}, \sigma)$ is $\Gamma$-invariant. 
Therefore, $B(g_{\CP^1},\sigma)$ induces a quadratic differential on $X$, which we will denote by $B(\sigma)$.

In summary of the above, we have the following Gaussian and mean curvatures of the Epstein surfaces in $M$.
\begin{lem}
\label{curvature-epstein}
The Gaussian curvature for the Epstein surface $\mathrm{Ep}_\sigma : X \to M$ is given by
\[
K(I(\sigma))
= \frac{4K(\sigma)}{(1-K(\sigma))^2 - \frac{16}{\sigma^2} |B(\sigma)|^2},
\]
and the mean curvature by 
\[
\pushQED{\qed}
H(\mathrm{Ep}_\sigma)
= \frac{K(\sigma)^2 - 1 - \frac{16}{\sigma^2}|B(\sigma)|^2}{(K(\sigma) - 1)^2 - \frac{16}{\sigma^2}|B(\sigma)|^2}.
\qedhere
\popQED
\]
\end{lem}
These are now equations on the compact Riemann surface $X$.

%*****************************************************
\section{Asymptotically Poincar\'e Families of Surfaces} 
\label{asym-def-section}
%*****************************************************

Previously, we have discussed Epstein surfaces for domains and for quotients. 
Generalizing this, we say an embedded closed surface $i : S \to M$ of genus greater than one is an \textit{Epstein surface of $X$} if there exists a $\Gamma$-invariant conformal metric $\tilde{\sigma}$ on $\Omega$ and a diffeomorphism $\varphi : X \to S$ such that the diagram 
\[
\begin{tikzpicture}[scale=0.7]
\node (1) at (0,2) {$\Omega$};
\node (2) at (3,2) {$\H^3$};
\node (3) at (0,0) {$S$};
\node (4) at (3,0) {$M$};

\draw[->] (1) to node [above] {$\mathrm{Ep}_{\tilde{\sigma}}$} (2);
\draw[->] (1) to node [left] {$\varphi \circ \pi_X$} (3);
\draw[->] (3) to node [above] {$i$} (4);
\draw[->] (2) to node [right] {$\pi_M$} (4);

\end{tikzpicture}
\]
commutes. Here the $\pi_X$ and $\pi_M$ are the respective quotient maps $\Omega \to X$ and $\H^3 \to M$.
When $S$ is an Epstein surface of $X$, the conformal metric $\tilde{\sigma}$ induces a conformal metric $\sigma$ on $X$ that we call the \textit{conformal metric at infinity}. 
We have $\mathrm{Ep}_{\sigma} = \varphi \circ i$. 
The Epstein surface $S$ will then refer to the embedded image of $\mathrm{Ep}_{\sigma}: X \to M$.

\begin{defn}
\label{asym-def}
Let $S_\epsilon$ for $\epsilon$ in $(0,1)$ be a family of Epstein surfaces with conformal metrics at infinity $\sigma(\epsilon)$. 
We call this family \textit{asymptotically Poincar\'e} if 
\begin{enumerate}
    \item there exists a scaling function $f:[0,1) \to [0,\infty)$ so that the path 
    \[
f\sigma:(0,1) \to \mathrm{Met}^\infty(X)
\]
is differentiable and converges to the hyperbolic metric on $X$ as $\epsilon \to 0$, that is, $f(\epsilon)\sigma(\epsilon) \to h$ as $\epsilon \to 0$,
    \item the function $f$ is smooth and has simple zero at 0, and
    \item the continuous extension $\gamma:[0,1) \to \mathrm{Met}^\infty(X)$ of $f \sigma$ is differentiable.
\end{enumerate}

\end{defn}

As will be shown, as $\epsilon$ tends towards zero, the family of surfaces is leaving the end of the manifold.
The following lemma shows that the surfaces in an asymptotically Poincar\'e family are asymptotically parallel.

\begin{lem}\label{asym-parallel-lemma}
Suppose $S_\epsilon$ is an asymptotically Poincar\'e family of surfaces. Let $t > 0$. If $g^t$ is the geodesic flow operator defined above, then  
\[
d_M \left(g^t (\widehat{\mathrm{Ep}}_{\sigma(\epsilon)}(z)), \mathrm{Ep}_{\sigma(e^{-2t} \epsilon)}(z) \right) \to 0
\quad \text{ as } \epsilon \to 0
\]
uniformly in $z$. 
That is, the distance between the surface $S_\epsilon$ flowed for time $t$ and the surface $S_{e^{-2t}\epsilon}$ tends towards zero as $\epsilon$ does. 
\end{lem}

\begin{proof}
We work with the universal covers. 
Lift $\sigma(\epsilon)$ to $\tilde{\sigma}(\epsilon)$ and $\gamma(\epsilon) = f(\epsilon)\sigma(\epsilon)$ to $\tilde{\gamma}(\epsilon)$ on $\Omega$. 
Write $\tilde{\sigma}(\epsilon) = e^{2\eta(\epsilon)}|dz|^2$ and $\tilde{\gamma}(\epsilon) = e^{2\lambda(\epsilon)}|dz|^2$, then $\eta(\epsilon) = \lambda(\epsilon) - (1/2)\ln(f(\epsilon))$. 
Recall from Lemma \ref{epstein-flow} that $g^t \circ \widehat{\mathrm{Ep}}_{\tilde{\sigma}(\epsilon)} = \mathrm{Ep}_{e^{2t}\tilde{\sigma}(\epsilon)}$.
For ease of notation, let $c = e^{-2t}$. 
Then a straightforward calculation using Proposition \ref{dumas-def} gives the distance between $\mathrm{Ep}_{c^{-1} \tilde{\sigma}(\epsilon)}$ and $\mathrm{Ep}_{\tilde{\sigma}(c\epsilon)}$ to be
\[
2 \mathrm{arctanh} \left(\sqrt{ 
\frac
{(1 - \sqrt{\frac{c f(\epsilon)}{f(c \epsilon)}}e^{\lambda(c\epsilon) - \lambda(\epsilon)})^2 + 4 c f(\epsilon)e^{-2\lambda(\epsilon)}|\lambda_z(c\epsilon) - \lambda_z(\epsilon)|^2}
{(1 + \sqrt{\frac{c f(\epsilon)}{f(c \epsilon)}}e^{\lambda(c\epsilon) - \lambda(\epsilon)})^2 + 4 c f(\epsilon)e^{-2\lambda(\epsilon)}|\lambda_z(c\epsilon) - \lambda_z(\epsilon)|^2}
} \right).
\] 
Since $\tilde{\gamma}(\epsilon)$ converges in $\mathrm{Met}^\infty(X)$, the function $\lambda$ has a $C^2$ limit as $\epsilon \to 0$. 
Therefore, the argument of $\mathrm{arctanh}$ converges to zero uniformly in $z$.
Thus the same holds in the quotient $M$ as well.
\end{proof}

We have required that an asymptotically Poincar\'e family consist of embedded surfaces.
The next proposition gives a useful condition for a family of conformal metrics to give rise to an asymptotically Poincar\'e family of surfaces.

\begin{prop}
\label{asym-family-prop}
Let $\sigma: (0,1) \to \mathrm{Conf}^\infty(X)$ be a family of conformal metrics on $X$. 
Suppose there exists a smooth function $f:[0,1) \to [0,\infty)$ with simple zero at $0$, such that $f\sigma \to h$ as $\epsilon \to 0$ and such that the extension $\gamma: [0,1) \to \mathrm{Conf}^\infty(X)$ is differentiable.
Then there exists an $\epsilon_0 >0$ so that for $\epsilon < \epsilon_0$, the Epstein map $\mathrm{Ep}_{\sigma(\epsilon)}$ is an embedding. 
Hence, the Epstein surfaces $\mathrm{Ep}_{\sigma(\epsilon)}:X \to M$, for $\epsilon < \epsilon_0$, form an asymptotically Poincar\'e family. 
\end{prop}

\begin{proof} 
Let $\tilde{\sigma}:(0,1) \to \mathrm{Conf}^\infty(\Omega)$ be the lift of the family $\sigma$. 
Define the Epstein family map $\mathrm{Ep}_{\tilde{\sigma}}: \Omega \times (0,1) \to \H^3$ by $\mathrm{Ep}_{\tilde{\sigma}}(z,\epsilon) = \mathrm{Ep}_{\tilde{\sigma}(\epsilon)}(z)$. 
It follows from the $\mathrm{SL}(2,\C)$-frame definition of the Epstein map that in the upper half space model $\H^3 \cong \C \times \R^+$, the family map is given by 
\[
\mathrm{Ep}_{\tilde{\sigma}}(z,\epsilon) = (z,0) + \frac{2}{e^{2\eta} + 4 |\eta_z|^2}\left(2 \eta_{\bar{z}}, e^\eta \right).
\]
Writing $\tilde{\sigma}(\epsilon) = e^{2 \eta(z,\epsilon)}|dz|^2$ and $\tilde{\gamma}(\epsilon) = e^{2 \lambda(z,\epsilon)}|dz|^2$, the condition $\tilde{\gamma}(\epsilon) = f(\epsilon)\tilde{\sigma}(\epsilon)$ becomes $\eta(z,\epsilon) = \lambda(z,\epsilon) - (1/2) \ln(f(\epsilon))$.
Note that $\lambda(z,\epsilon) \to \rho(z)$ as $\epsilon \to 0$, uniformly in $z$, where $\rho$ is the log density of the Poincar\'e metric of $\Omega$. 
Hence we can rewrite $\mathrm{Ep}_{\tilde{\sigma}}$ as 
\[
\mathrm{Ep}_{\tilde{\sigma}}(z,\epsilon) = (z,0)  + \frac{2}{e^{2\lambda} + 4 f(\epsilon) |\lambda_z|^2} \left( 2 f(\epsilon) \lambda_{\bar{z}},  \sqrt{f(\epsilon)}e^{\lambda} \right)
\]
and see that
\[
\lim_{\epsilon \to 0} \mathrm{Ep}_{\tilde{\sigma}} (z,\epsilon) = (z,0).
\]
So we may extend $\mathrm{Ep}_{\tilde{\sigma}}$ to a map $\Omega \times [0,1) \to \H^3 \sqcup \Omega$, which is the identity on the boundary $\Omega \times \{0\} \to \Omega$. 

While this map is not differentiable at $\epsilon = 0$ (due to the $\sqrt{f(\epsilon)}$), the map $F: \Omega \times [0,1) \to \H^3 \sqcup \Omega$ given by $F(z,\epsilon) = \mathrm{Ep}_{\tilde{\sigma}}(z,\epsilon^2)$ satisfies $F(z,0) = (z,0)$, is differentiable at $\epsilon = 0$, and has derivative  
\[
d F_{(z,0)} = 
\begin{pmatrix}
1 & 0 & 0 \\
0 & 1 & 0 \\
0 & 0 & 2 \sqrt{f'(0)} e^{-\rho(z)}
\end{pmatrix}.
\]
Since this is invertible, $F$ is a local $C^1$-diffeomorphism at the boundary of $\Omega \times [0,1) \to \H^3 \sqcup \Omega$.

Define $\bar{M} = ( \H^3 \sqcup \Omega ) /\Gamma = M \sqcup X$.
Then $\bar{M}$ is a smooth manifold with compact boundary $X$. 
By $\Gamma$-equivariance of $\mathrm{Ep}_{\tilde{\sigma}}$, $F$ descends to a map $X \times[0,1) \to \bar{M}$ that is the identity on the boundary $\partial(X \times [0,1)) \to \partial{\bar{M}} = X$ and that is a local diffeomorphism there. 
Therefore, the restriction of $F$ to $X \times [0, \delta)$, for some small enough $\delta$, is a diffeomorphism onto a collar neighborhood of $\partial \bar{M} = X$.

Unraveling, we get that $\mathrm{Ep}_{\sigma}$ is a diffeomorphism from a collar neighborhood $X \times (0,\sqrt{\delta})$ to a neighborhood of infinity of $M$.
In particular, each Epstein map $\mathrm{Ep}_\sigma(\cdot, \epsilon) = \mathrm{Ep}_{\sigma(\epsilon)}$, for $\epsilon < \sqrt{\delta}$, is an immersion and injective with compact domain $X$. 
Hence each Epstein surface, for $\epsilon < \sqrt{\delta}$, is embedded.
To complete the proof take $\epsilon_0 = \sqrt{\delta}$.
\end{proof}

In the preceeding proof, we have the map $F: X \times [0,\delta) \to \bar{M}$ is a diffeomorphism onto its image. Hence we have the following result.

\begin{cor}
\label{foliation}
If $S_\epsilon$ is an asymptotically Poincar\'e family of surfaces, then there exists an $\epsilon_0 > 0$ such that for $\epsilon < \epsilon_0$ the surfaces $S_\epsilon$ form a foliation of the end of $M$ whose surface at infinity is $X$.
\qed
\end{cor}

We now turn to our main result, but first note that the co-orientation on an Epstein surface we are using is that induced by the lift $\widehat{\mathrm{Ep}}_\sigma$, which points towards the surface at infinity $X$.
This implies that $I\!I(\sigma)$ is negative definite (for small enough $\epsilon$), and so $-I\!I(\sigma)$ is a smooth Riemannian metric. 

In our setting it is natural to work with the Riemannian model of Teichm\"uller space. 
That is, we use 
\[
\mathcal{T}(X) = \mathrm{Met}^\infty(X)/  \mathrm{Diff}_0^\infty(X) \rtimes P^\infty(X),
\]
where $P^\infty(X)$ is the set of smooth positive functions on $X$ and $\mathrm{Diff}_0^\infty(X)$ is the group of smooth diffeomorphisms isotopic to the identity. 
See \cite{tromba1992} for details. 
The smooth topology, however, is difficult to work with directly. 
So, we first work in the Sobolev setting $\mathrm{Met}^s(X)$ and $H^s(X)$ of Sobolev tensors and functions of a fixed regularity $s > 3$.
Since $K$ and $B$ are smooth functions of $\sigma$ and its derivatives we have that they both extend to functions on Sobolev classes of metrics.
Hence if $\sigma \in \mathrm{Conf}^s(X)$ then $I(\sigma)$ and $-I\!I(\sigma)$ belongs to $\mathrm{Met}^{s-2}(X)$.
We will obtain results with these extensions and then argue our results are independent of the chosen $s$.

\begin{prop}
\label{thm-in-sobolev}
Suppose $S_\epsilon$ is an asymptotically Poincar\'e family of surfaces. Let $\gamma : [0,1) \to \mathrm{Met}^s(X)$ be the extension of $f\sigma$ thought of as taking values in a class of Sobolev metrics for a fixed $s > 3$. Then the first and second fundamental forms $I \circ \sigma: (0,1) \to \mathrm{Met}^{s-2}(X)$ and  $I\!I \circ \sigma: (0,1) \to \mathrm{Met}^{s-2}(X)$ satisfy
\[
I_\epsilon = 4 f'(0) \epsilon I(\sigma(\epsilon)) \to h
\quad \text{ and } \quad
I\!I_\epsilon =  - 4 f'(0) \epsilon I\!I(\sigma(\epsilon)) \to h
\quad \text{ as } \epsilon \to 0.
\]
Moreover, $I_\epsilon$ and $I\!I_\epsilon$ are differentiable at $\epsilon = 0$ and their tangent vectors are given by 
\[
\dot{I_\epsilon} = \dot{\gamma} + 2 f'(0) h + 4 f'(0) \mathrm{Re}(\phi)
\quad \text{ and } \quad
\dot{I\!I_\epsilon} = \dot{\gamma}
\]
\end{prop}

\begin{proof}
We have that $\gamma:[0,1) \to \mathrm{Met}^\infty(X)$ is continuous and differentiable at $\epsilon = 0$. Therefore, as $\mathrm{Met}^\infty(X) = \cap_{s > 3} \mathrm{Met}^s(X)$ we also have that $\gamma$ is continuous to $\mathrm{Met}^s$ and differentiable at $\epsilon = 0$. 

Then, we have $I(\sigma(\epsilon)) = I(\frac{1}{f(\epsilon)} \gamma(\epsilon))$ is equal to 
\[
4 f(\epsilon) \frac{|B(\gamma(\epsilon))|^2}{\gamma(\epsilon)} + \frac{1}{4 f(\epsilon)}(1 - f(\epsilon) K(\gamma(\epsilon)))^2 \gamma(\epsilon) + 2(1 - f(\epsilon)K(\gamma(\epsilon)))\mathrm{Re}(B(\gamma(\epsilon))),
\]
which is a smooth tensor independent of $s$. Since $f:[0,1) \to \R$ and $\gamma:[0,1) \to \mathrm{Met}^s(X)$ are differentiable at 0 and since $K: \mathrm{Met}^s(X) \to H^{s-2}(X)$ and $B: \mathrm{Conf}^s(X) \to \Gamma^{s-2}(\Sigma^{2}(X))$ are differentiable at the hyperbolic metric $h$, we can write 
\begin{align*}
f(\epsilon) = \epsilon f'(0) + O(\epsilon^2), \quad
K(\gamma(\epsilon)) = -1 + O(\epsilon), \quad 
B(\gamma(\epsilon)) = \frac{1}{2} \phi + O(\epsilon).
\end{align*}

Substitution and some simplification gives 
\[
I(\sigma(\epsilon)) = \frac{1}{4\epsilon f'(0)} h + \frac{1}{4 f'(0)} \dot{\gamma} + \frac{1}{2} h + \mathrm{Re}(\phi) + O(\epsilon).
\]
Consequently, 
\[
I_\epsilon =  h + \epsilon(\dot{\gamma} + 2 f'(0)h + 4 f'(0) \mathrm{Re}(\phi)) + O(\epsilon^2).
\]
The same reasoning will give $I\!I_\epsilon = h + \epsilon \dot{\gamma} + O(\epsilon^2)$. The result then follows.
\end{proof}

The space $\mathrm{Met}^s(X)$ is an open subset of the vector space of symmetric 2-tensors $\Gamma^s(\Sigma^2(X))$. So, we may canonically identify each tangent space to $\mathrm{Met}^s(X)$ with this space. For each $g \in \mathrm{Met}^s(X)$, there is an $L^2$ orthogonal decomposition of the tangent space at $g$ as
\[
T_g \mathrm{Met}^s(X) = \{\dot{g} \ | \ \mathrm{tr}_g (\dot{g}) = 0 = \delta_g(\dot{g}) \} \oplus \{\mathcal{L}_X g + fg \ |\ f \in H^s \text{ and } X \in \Gamma(TX) \},
\] 
where $\delta_g(\dot{g})$ is the divergence of the 1-1 tensor $g^{-1}\dot{g}$. The second summand is tangent to the $\mathrm{Diff}_0^s(X) \rtimes P^s(X)$ orbit of $g$. The first summand is referred to as transverse traceless tensors and consists of elements that are smooth, trace-free, and divergence-free (see \cite{fischer-marsden1975} for more details). Since the transverse traceless tensors are orthogonal to the group orbit they may be identified with the tangent space $T_{[g]} \left( \mathrm{Met}^s(X)/ \mathrm{Diff}_0^s(X) \rtimes P^s(X) \right)$ to the quotient at $[g]$. Moreover transverse traceless tensors are a set of smooth tensors that are $L^2$ orthogonal to the $\mathrm{Diff}_0^s(X) \rtimes P^s(X)$ orbit \textit{for any} $s$. Consequently, they may be identified with the tangent space to the quotient $\mathrm{Met}^\infty(X)/(\mathrm{Diff}_0^\infty(X) \rtimes P^\infty(X))$. That is, they may be naturally identified with the tangent space to Teichm\"uller space at $[g]$:
\[
T_{[g]} \mathcal{T}(X) = \{\dot{g} \ | \ \mathrm{tr}_g (\dot{g}) = 0 = \delta_g(\dot{g}) \}.
\]
Under this identification, the action of the derivative of the projection $\pi: \mathrm{Met}^\infty(X) \to \mathcal{T}(X)$ at $g$ is given by orthogonal projection onto the transverse traceless tensors $T_g \mathrm{Met}^\infty(X) \to \{\dot{g} \ | \ \mathrm{tr}_g (\dot{g}) = 0 = \delta_g(\dot{g}) \}.$

We make one more comment before we  prove our main Theorem. 
\begin{lem}[See \cite{tromba1992}]
Let $X$ be a Riemann surface and $h$ the corresponding hyperbolic metric. If $\phi$ is a holomorphic quadratic differential, then $\mathrm{Re}(\phi)$ is a smooth $h$-transverse-traceless tensor. 
\end{lem} 

We now prove our main result, Theorem \ref{big-thm-intro} from the Introduction. Recall that $[I\!I]$ denotes the point in Teichm\"uller space corresponding to the conformal class of $-I\!I$.

\begin{thm}\label{main-result}
Let $S_\epsilon$ for $\epsilon \in (0,1)$ be an asymptotically Poincar\'e family of surfaces with metrics at infinity $\sigma(\epsilon)$. If $h$ is the hyperbolic metric of $X$ and $\phi$  the holomorphic quadratic differential at infinity, then in Teichm\"uller space $\mathcal{T}(X)$ we have 
\[
[I(\sigma(\epsilon))] \to [h]
\quad \text{ and } \quad
[I\!I(\sigma(\epsilon))] \to [h]
\quad \text{ as } \epsilon \to 0.
\]
Moreover, the tangent vectors in $T_{[h]} \mathcal{T}(X)$ are given by 
\[
\dot{[I(\sigma(\epsilon))]}  = 4 f'(0) \mathrm{Re}(\phi) \quad \text{ and } \quad \dot{[I\!I(\sigma(\epsilon))]} = 0.
\]
\end{thm}

\begin{proof}
Using the notation from  Proposition \ref{thm-in-sobolev}, for each $s > 3$ we have that $I_\epsilon$ and $I\!I_\epsilon$, which are paths through smooth tensors,  converge in $\mathrm{Met}^s(X)$ to the hyperbolic metric $h$. 
Since $\mathrm{Met}^\infty(X)  = \cap \mathrm{Met}^s(X)$ we know that $I_\epsilon$ and $I\!I_\epsilon$ converge to $h$ in $\mathrm{Met}^\infty(X)$. 
Moreover, since the projection $\mathrm{Met}^\infty(X) \to \mathcal{T}(X)$ is continuous, and since $[I(\sigma(\epsilon))] = [I_\epsilon]$ and $[I\!I(\sigma(\epsilon))] = [I\!I_\epsilon]$, we have that
\[
[I(\sigma(\epsilon))] \to [h]
\quad \text{ and } \quad 
[I\!I(\sigma(\epsilon))] \to [h]
\quad \text{ as } \epsilon \to 0.
\]
Since both paths converge to $[h]$ at $\epsilon = 0$ we can extend them to continuous paths $[0,1) \to \mathcal{T}(S)$. Since $[I(\sigma(\epsilon))]$ and  $[I\!I(\sigma(\epsilon))]$ agree with $[I_\epsilon]$ and $[I\!I_\epsilon]$, respectively, we also have they are differentiable at $\epsilon = 0$.

From Proposition \ref{thm-in-sobolev} we know that 
\[
\dot{I_\epsilon}  = \dot{\gamma} + 2 f'(0) h + 4 f'(0) \mathrm{Re}(\phi) \quad \text{ and } \quad \dot{I\!I_\epsilon} = \dot{\gamma}.
\]
Since $\phi$ is holomorphic, $4 f'(0) \mathrm{Re}(\phi)$ is trace-free and divergence-free. On the other hand, $2 f'(0) h$ is pure trace and so belongs to the $\mathrm{Diff}_0^s(X) \rtimes P^s(X)$ orbit of $h$. Furthermore, for any $s$ it belongs to the group orbit of $h$. Hence, the derivative of the projection at $h$ removes this term.
Similarly, $\gamma$ is a path in the group orbit of $h$ and so $\dot{\gamma}$ projects to 0. 
We then have
\begin{align*}
\dot{[I(\sigma(\epsilon))]}
= \dot{[I_\epsilon]}
= d \pi_h (\dot{\gamma} + 2 f'(0) h + 4 f'(0) \mathrm{Re}(\phi)) 
= 4 f'(0) \mathrm{Re}(\phi)
\end{align*}
and
\[
\dot{[I\!I(\sigma(\epsilon))]} = \dot{I\!I_\epsilon} = d \pi_h (\dot{\gamma}) = 0,
\]
as claimed.
\end{proof}

%*************************************************
\section{$k$-Surfaces in Quasi-Fuchsian Manifolds} \label{k-surfaces-section}
%*************************************************

\subsection{The $k$-Surface Equation}

We now prove that $k$-surfaces form an asymptotically Poincar\'e family of surfaces. 
To do this we derive an equation for a conformal metric that implies its Epstein surface is a $k$-surface and we show this equation has a unique smooth solution for each $k$ near zero. 
In the proof of existence of solutions we will see that these conformal metrics satisfy the hypothesis of Proposition \ref{asym-family-prop}. 
This also gives an alternative proof to Labourie's theorem on the existence of $k$-surfaces, in this case for quasi-Fuchsian manifolds.

Finding a metric $\sigma$ whose Epstein surface has constant Gaussian curvature $k$ is finding a metric $\sigma$ that solves $K(I(\sigma)) = k$, which from Lemma \ref{curvature-epstein} is solving
\[
\frac{4K(\sigma)}{(1-K(\sigma))^2 - \frac{16}{\sigma^2} |B(\sigma)|^2} = k.
\]
For now we will focus on solving
\begin{equation}
\label{k-surface-equation}
4K(\sigma) = k \left((1-K(\sigma))^2 - \frac{16}{\sigma^2} |B(\sigma)|^2 \right),
\end{equation}
as we will see that for small enough $k$, $K(I(\sigma)) = k$.

We are interested in obtaining solutions to (\ref{k-surface-equation}) for $k$ near zero. 
This is hampered by the fact that there are no solutions to (\ref{k-surface-equation}) when $k = 0$. 
Indeed we would be asking for $K(\sigma)=0$, which is impossible on a surface with genus bigger than 1. 
In an attempt to obtain better asymptotics we consider the case when $\Gamma$ is Fuchsian. 
Working in the universal covers, $\mathrm{Ep}_h : \Omega \to \H^3$ gives a totally geodesic copy of $\H^2$ in $\H^3$ with constant curvature $-1$. 
For $-1<k<0$  the $k$-surfaces are given by equidistant copies of this $\H^2$. 
The conformal metric whose Epstein surface is the $k$-surface is then $c(k)h$ for some function of $k$ satisfying $K(I(c(k)h)) = k$. 
Since $B(g_{\CP^1},c(k)h) = 0$ we get the defining equation of $c$ as $4K(c(k)h) = k(1 - K(c(k)h))^2$.
More explicitly $c$ is given by
\begin{equation}
\label{c-def}
c(k) = \frac{1+\sqrt{1+k}}{1-\sqrt{1+k}}.
\end{equation}

Suppose we have solution metrics $\sigma_k$  on $X$ that give Epstein surfaces with constant Gaussian curvature $k$. 
As we just saw, in the Fuchsian case we have $\sigma_k = c(k) h$.  
In the general quasi-Fuchsian case, define $f(k) = c(k)^{-1}$ and let $\tau_k = f(k)\sigma_k$, so that in the Fuchsian case the metric $\tau_k$ is the hyperbolic metric for all $k$. 
Since $\sigma_k$ solves (\ref{k-surface-equation}), by substitution and some simplification we get that $\tau = \tau_k$ solves the equation 
\begin{equation}
\label{scaled-equation}
(2+k)(1+K(\tau))^2 + 2\sqrt{1+k}\left(1-K(\tau)^2\right) + 16\left(2\sqrt{1+k} - 2 - k  \right)\frac{|B(\tau)|^2}{\tau^2} = 0.
\end{equation}

In the limiting case $k=0$ we see that the hyperbolic metric $h$ on $X$ solves (\ref{scaled-equation}). 
Solutions actually exists in a neighborhood of $(k,\tau) = (0,h)$ as we will show.
To apply PDE theory we define the function $F : U \times \mathrm{Conf}^\infty(X) \to C^\infty(X)$ given by 
\begin{multline*}
F(k,\tau) = (2+k)(1+K(\tau))^2  \\ + 2\sqrt{1+k}\left(1-K(\tau)^2 \right) +16\left(2\sqrt{1+k} - 2 - k  \right)\frac{|B(\tau)|^2}{\tau^2}.
\end{multline*}
Here, $U$ is an open interval around $0$ small enough not to contain $-1$ so that $F$ is smooth on its domain. 
Points $(k,\tau)$ where $F(k,\tau) = 0$ are solutions to the scaled equation (\ref{scaled-equation}); we have already found $F(0,h) = 0$.
 
\subsection{Solutions to the $k$-Surface Equation}

We will use the Implicit Function Theorem in the Banach space setting to obtain solutions and so we first work with $\mathrm{Conf}^s(X) = \{ p h \ |\ p \in H^s(X,\R^+)\}$, where $H^s(X,\R^+)$ is the Sobolev space of functions on $X$ taking positive values.
With the norm $\|\tau\|_s := \| \tau/h\|_s$, the set $\mathrm{Conf}^s(X)$ is naturally identified with an open subset of the Banach space $H^s(X) = H^{s}(X,\R)$.
Fix $s>3$ from now on. 
Extend $F$ to a function $U \times \mathrm{Conf}^s(X) \to H^{s-2}(X)$ so that $F$ is then defined on an open subset of a Banach space.

\begin{thm}
\label{weak-solutions}
There is a neighborhood $V$ of $0$ such that for each $k \in V$, there exists a unique $\tau \in \mathrm{Conf}^s(X)$ such that $F(k,\tau) = 0$.
\end{thm}

\begin{proof} 
Note that since the constituent parts of $F$ are smooth on $U$, $F$ is smooth. 
Furthermore, $D_2F_{(0,h)} : H^s(X) \to H^{s-2}(X)$ is an isomorphism, where $D_2F$ is the partial derivative of $F$ with respect to its second argument. 
Indeed, a direct computation of the derivative of $F$ at $(0,h)$ yields 
\[
DF_{(0,h)}(\dot{k},\dot{\tau}) = 4 \, D K_h(\dot{\tau}).
\]
Note that $D_1F_{(0,h)} = 0$ and that $D_2F_{(0,h)} = 4 \, DK_h$. 
The differential of the curvature function $D K_h$ is given by a formula of Lichnerowicz: 
\[
 4\ DK_h(\dot{\tau}) = -2(\Delta_h - Id)\frac{\dot{\tau}}{h},
\]
which is an isomorphism $H^s(X) \to H^{s-2}(X)$ (see \cite[Page 33]{tromba1992}).

Consequently, by the Banach Implicit Function Theorem (see \cite[Theorem 17.6]{gilbarg-trudinger2001}) there exists a neighborhood $V$ of $0$ and a curve $\gamma : V \to \mathrm{Conf}^s(X)$ with $\gamma(0) = h$ and $F(k, \gamma(k)) = 0$. 
Moreover, these are the only solutions to $F(k,\tau) = 0$ in $V$, and $\gamma$ is smooth since $F$ is.
\end{proof}

\subsection{Regularity of Solutions}
Theorem \ref{weak-solutions} furnishes weak solutions $\gamma(k)$ that vary smoothly as a map $V \to \mathrm{Conf}^s(X)$. 
The Sobolev Embedding Theorem immediately strengthens the regularity of the individual $\gamma(k)$ to at least $C^{2,\alpha}$, for each $k$ (see \cite{aubin1982}), and so we have strong solutions $F(k,\gamma(k)) = 0$.

A nonlinear equation $A(u) = 0$ is said to be elliptic at $u$ if the derivative of $A$ at $u$, $DA_u$, is an elliptic linear operator. 
When $A$ is smooth, solutions $u$ where $A$ is elliptic are also smooth (\cite[Lemma 17.16]{gilbarg-trudinger2001}). 
In our case, since $D_2F_{(0,h)} = -2(\Delta_h - Id)$ is an elliptic operator we have $F(0, \tau) = 0$ is elliptic at $\tau = h$. 
Moreover, ellipticity is an open condition in $\R \times C^2(X)$, so there exists an open interval $(-\delta,\delta)$ such that the linearization $D_2F(k,\gamma(k))$ is elliptic for all $k \in (-\delta,\delta)$.
This has two consequences.
First, it shows the solutions given by Theorem \ref{weak-solutions} are individually smooth conformal metrics.
We conclude the following theorem.

\begin{thm}
\label{k-surfaces-existence}
There exists an $\delta > 0$ such that for all $-\delta < k < 0$ there exists a unique smooth metric $\sigma_k$ whose Epstein surface is a $k$-surface.
\end{thm}

\begin{proof}
We have a curve $\gamma$ defined on an interval $(-\delta, \delta)$ such that the smooth metric $\gamma(k)$ satisfies $F(k,\gamma(k)) = 0$. 
This means that $\gamma(k)$ solves the scaled equation (\ref{scaled-equation}) and so, for $k \in (-\delta,0)$, the metric $\sigma(k) = c(k) \gamma(k)$ (where $c(k)$ is defined by (\ref{c-def})) satisfies the $k$-surface equation (\ref{k-surface-equation}).  
Finally, we must show that $(1-K(\sigma))^2 - \frac{16}{\sigma^2} |B(\sigma)|^2$ in (\ref{k-surface-equation}) is everywhere nonzero so that (\ref{k-surface-equation}) is equivalent to the desired Guassian curvature condition.
But this follows from Nehari's Theorem (see \cite[Theorem 1.3]{lehto-1987}, or the original paper \cite{nehari-1949}), which gives an a priori bound on $|B(h)|^2/h^2$, implying that $|B(\sigma_k)|^2/\sigma_k^2 \to 0$ as $k \to 0$.
\end{proof}

Second, ellipticity also implies the family $\gamma(k)$ varies smoothly in $k$ in the $C^\infty$ topology. 
Indeed, it implies the Banach Implicit Function Theorem may be applied to $D_2F(k,\gamma(k)): H^t(X) \to H^{t-2}(X)$ for any $t > s$ and any $k \in (-\delta,\delta)$ to get that $\gamma$ is also a smooth path into $\mathrm{Conf}^t(X)$.
We have the following corollary.

\begin{cor}
\label{k-surfaces-parallel}
The family of $k$-surfaces forms an asymptotically Poincar\'e family. 
\label{k-surfaces-cor}
\end{cor}

\begin{proof}
We show that $\tilde{\sigma}(\epsilon) = \sigma_{-\epsilon}$ satisfies the hypotheses of Proposition \ref{asym-family-prop}. 
Since $\mathrm{Conf}^\infty(X) = \cap \mathrm{Conf}^s(X)$ and since $\gamma$ is smooth on $(-\delta,\delta)$ into $\mathrm{Conf}^s(X)$ for all $s > 3$, we have the $\gamma$ is smooth into $\mathrm{Conf}^\infty(X)$.
Now, define $\tilde{f}(\epsilon) = f(- \epsilon)$ and $\tilde{\gamma}(\epsilon) = \gamma(-\epsilon)$ for $\epsilon \in (0,\delta)$. 
Then the Corollary follows from Proposition \ref{asym-family-prop} since $\tilde{f}$ is smooth on $[0,\delta)$ with derivative $\tilde{f}'(0) = 1/4$, since $\tilde{\gamma}$ is differentiable, and since $\tilde{f}(\epsilon) \tilde{\sigma}(\epsilon) = \tilde{\gamma}(\epsilon) = \gamma(k) \to h$ as $\epsilon \to 0$.
\end{proof}

\subsection{A Conjecture of Labourie}

A geometrically finite end of a hyperbolic 3-manifold admits a foliation by $k$-surfaces. 
Labourie first proved their existence in \cite{labourie1991}. 
In \cite{labourie1992} Labourie discusses how the conformal classes $[I_k]$ and $[I\!I_k]$ behave as paths in the Teichm\"uller space of $X$. 
When $k \to -1$, $[I\!I_k]$ approaches the point at infinity of $\mathcal{T}(X)$ corresponding to the measured geodesic lamination on the convex core of the quasi-Fuchsian manifold. 
When $k \to 0$, both $[I_k]$ and $[I\!I_k]$ converge to the same point: the complex structure $X$ on the surface at infinity.
He conjectures that the tangent vectors to these paths are related to the holomorphic quadratic differential at infinity $\phi$. 

Our Theorem \ref{k-surfaces-existence} is another proof of the existence of these $k$-surfaces, at least for $k$ close to zero. 
By the uniqueness result of Theorem 1.10 in \cite{labourie1992}, given an end of $M$ and $k \in (-1,0)$, there exists a unique immersed incompressible $k$-surface. 
Since Epstein surfaces are incompressible, the $k$-surfaces produced by our Theorem \ref{k-surfaces-existence} are the $k$-surfaces Labourie obtains. 
Our approach here is more concrete than the methods used by Labourie. 
In return for sacrificing the generality of Labourie's pseudoholomorphic methods, a proof of Labourie's conjecture follows easily. 
Corollary \ref{k-surfaces-parallel} gives that, for $k$ near zero, these $k$-surfaces form an asymptotically Poincar\'e family, and so Theorem \ref{main-result} resolves Labourie's conjecture.

\begin{thm}
Let $I_k$ and $I\!I_k$ be the first and second fundamental forms of the $k$-surface. 
Let $\phi$ be the holomorphic quadratic differential at infinity of $M$. 
Then, as $k \to 0$, the tangent vectors to $[I_k]$ and $[I\!I_k]$ in Teichm\"uller space are given by 
\[
\dot{[I_k]} = - \mathrm{Re}(\phi) \quad \text{and } \quad   \dot{[I\!I_k]} = 0.
\]
\end{thm}

\begin{proof}
We take $\epsilon = -k$. The derivative $f'(0) = -1/4$ can be computed directly. 
Theorem \ref{main-result} now gives the theorem.
\end{proof}

%*********************************************************************
\section{Constant Mean Curvature Surfaces in Quasi-Fuchsian Manifolds} \label{mean-curvature-section}
%*********************************************************************
We now prove that there exists an asymptotically Poincar\'e family of surfaces $S_k$ for $k$ near zero, such that the mean curvature of the surface $S_k$ is $-\sqrt{1+k}$. 
This is done similarly to the previous section: we derive an equation for a conformal metric that implies its Epstein surface has mean curvature $-\sqrt{1+k}$ and we show this equation has a unique smooth solution for each $k$ near zero. 
The proof of existence of solutions will show that these conformal metrics satisfy the hypothesis of Proposition \ref{asym-family-prop}.

Recall from Lemma \ref{curvature-epstein} that the mean curvature of $\mathrm{Ep}_\sigma: X \to M$ is given by 
\[
H(\mathrm{Ep}_\sigma)
= \frac{K(\sigma)^2 - 1 - \frac{16}{\sigma^2}|B(\sigma)|^2}{(K(\sigma) - 1)^2 - \frac{16}{\sigma^2}|B(\sigma)|^2}.
\]
To find a metric whose Epstein surface has constant mean curvature $-\sqrt{1+k}$ we must solve the equation $H(\mathrm{Ep}_\sigma) = -\sqrt{1+k}$, which simplifies to
\begin{equation}
\label{mean-curvature-equation}
1-K(\sigma)^2 - \sqrt{1+k}(1-K(\sigma))^2 + (1 + \sqrt{1+k})\frac{16}{\sigma^2}|B(\sigma)|^2 = 0.
\end{equation}
As in the $k$-surfaces case it suffices to solve (\ref{mean-curvature-equation}) since $(K(\sigma) - 1)^2 - \frac{16}{\sigma^2}|B(\sigma)|^2$ will eventually be nonzero.
Furthermore, we scale the equation by assuming $\sigma_k$ solves (\ref{mean-curvature-equation}) and defining $\tau_k = f(k) \sigma_k$ for $f(k) = \frac{1-\sqrt{1+k}}{1+\sqrt{1+k}}$. 
If $\sigma_k$ solves (\ref{mean-curvature-equation}) then $\tau = \tau_k$ solves $G(k,\tau) = 0$ for $G: U \times \mathrm{Conf}^\infty(X) \to C^\infty(X)$ defined by 
\[
G(k,\tau) = 1+\sqrt{1+k} + 2\sqrt{1+k}K(\tau) + (-1 + \sqrt{1+k})(K(\tau)^2 - \frac{16}{\tau^2}|B(\tau)|^2).
\]
Here $U$ is a small enough open set around zero not containing $-1$ so that $G$ is smooth on its domain.

\begin{thm}
There exists a neighborhood $W$ of 0 so that for each $k \in W$, there exists a unique $\tau \in \mathrm{Conf}^\infty(X)$ so that $G(k,\tau) = 0$.
\end{thm}

\begin{proof}
Extend $G$ to a map $U \times \mathrm{Conf}^s(X) \to H^{s-2}(X)$ for $s > 3$. When $k = 0$ we have the hyperbolic metric $h$ as solution $G(0,h) = 0$. 
The map $G$ is smooth. 
The derivative of $G$ at $(0,h)$ is given by 
\[
dG_{(0,h))}(\dot{k},\dot{\tau}) = -4 \frac{|\phi|^2}{h^2} \dot{k} + 2 D K_h(\dot{\tau}).
\]
Notice that $D_2 G_{(0,h)} = D K_h$, which---as in the proof of Theorem \ref{weak-solutions}---is an isomorphism $H^s(X) \to H^{s-2}(X)$.
Hence, by the Banach Implicit Function Theorem there exists an open set $W$ and a smooth curve $\gamma: W \to \mathrm{Conf}^{s}(X)$ such that $G(k,\gamma(k)) = k$. Moreover, these are the only solutions to $G(k,\tau) = 0$ in $W$.

The same regularity arguments apply here as in the $k$-surface case and imply the existence of an $\delta > 0$ so that when $k \in (-\delta,\delta)$, each metric $\gamma(k)$ is smooth and the family $\gamma$ varies smoothly in $k$ in the $C^\infty$ topology. 
\end{proof}

This implies that the mean curvature surfaces form an asymptotically Poincar\'e family of surfaces. 

\begin{cor}
There exists a $\delta > 0$ so that for $k \in (-\delta, 0)$ the family of surfaces $S_k$ where $S_k$ has constant mean curvature $-\sqrt{1+k}$ forms an asymptotically Poincar\'e family of surfaces. 
\end{cor}

\begin{proof}
Since $\gamma(k)$ solves $G(k,\gamma(k)) = 0$ for $k \in (-\delta,\delta)$, the metric $\gamma(k)$ solves the scaled mean curvature equation. 
So, defining $\sigma_k = \frac{1 + \sqrt{1+k}}{1 - \sqrt{1+k}} \ \gamma(k)$ for $k \in (-\delta,0)$, we see that $\sigma_k$ solves the mean curvature equation (\ref{mean-curvature-equation}), which implies that $\mathrm{Ep}_{\sigma_k}$ has constant mean curvature $-\sqrt{1+k}$.

The same arguments as in Corollary \ref{k-surfaces-cor} show that these metrics satisfy the hypothesis of Proposition \ref{asym-family-prop}.
So, the surfaces $S_k = \mathrm{Ep}_{\sigma_{k}}(X)$ form an asymptotically Poincar\'e family of surfaces. 
\end{proof}

Consequently, we obtain Theorem \ref{cmc-intro}.

\providecommand{\bysame}{\leavevmode\hbox to3em{\hrulefill}\thinspace}
\providecommand{\MR}{\relax\ifhmode\unskip\space\fi MR }
\providecommand{\MRhref}[2]{%
  \href{http://www.ams.org/mathscinet-getitem?mr=#1}{#2}
}
\providecommand{\href}[2]{#2}


\begin{thebibliography}{{Sch}17}

\bibitem[And98]{anderson1998}
C.~Gregory Anderson, \emph{Projective structures on {R}iemann surfaces and
  developing maps to $\mathbb{H}^3$ and $\mathbb{C}\mathrm{P}^n$}, Ph.D.
  thesis, University of California, Berkeley, 1998.

\bibitem[Aub82]{aubin1982}
Thierry Aubin, \emph{Nonlinear analysis on manifolds. {M}onge-{A}mp{\`e}re
  equations}, Grundlehren der Mathematischen Wissenschaften, vol. 252,
  Springer-Verlag, New York, 1982.

\bibitem[Dum09]{dumas2009}
David Dumas, \emph{Complex projective structures}, Handbook of
  {T}eichm{\"u}ller theory. {V}ol. {II}, IRMA Lect. Math. Theor. Phys.,
  vol.~13, Eur. Math. Soc., Z{\"u}rich, 2009, pp.~455--508.

\bibitem[Dum17]{dumas2017}
\bysame, \emph{Holonomy limits of complex projective structures}, Adv. Math.
  \textbf{315} (2017), 427--473.

\bibitem[Eps84]{epstein1984}
Charles~L. Epstein, \emph{Envelopes of horospheres and weingarten surfaces in
  hyperbolic 3-space}, Preprint, 1984.

\bibitem[FM75]{fischer-marsden1975}
Arthur~E. Fischer and Jerrold~E. Marsden, \emph{Deformations of the scalar
  curvature}, Duke Math. J. \textbf{42} (1975), no.~3, 519--547.

\bibitem[GT01]{gilbarg-trudinger2001}
David Gilbarg and Neil~S. Trudinger, \emph{Elliptic partial differential
  equations of second order}, Classics in Mathematics, Springer-Verlag, Berlin,
  2001.

\bibitem[Lab91]{labourie1991}
Fran\c{c}ois Labourie, \emph{Probl{\`e}me de {M}inkowski et surfaces {\`a}
  courbure constante dans les vari{\'e}t{\'e}s hyperboliques}, Bull. Soc. Math.
  France \textbf{119} (1991), no.~3, 307--325.

\bibitem[Lab92]{labourie1992}
\bysame, \emph{Surfaces convexes dans l'espace hyperbolique et {${\mathbb
  C}{\rm P}^1$}-structures}, J. London Math. Soc. (2) \textbf{45} (1992),
  no.~3, 549--565.

\bibitem[Leh87]{lehto-1987}
Olli Lehto.
\newblock {\em Univalent functions and {T}eichm\"{u}ller spaces}, volume 109 of
  {\em Graduate Texts in Mathematics}.
\newblock Springer-Verlag, New York, 1987.


\bibitem[MP11]{mazzeo-pacard2011}
Rafe Mazzeo and Frank Pacard, \emph{Constant curvature foliations in
  asymptotically hyperbolic spaces}, Rev. Mat. Iberoam. \textbf{27} (2011),
  no.~1, 303--333.

\bibitem[Neh49]{nehari-1949}
Zeev Nehari.
\newblock The {S}chwarzian derivative and schlicht functions.
\newblock {\em Bull. Amer. Math. Soc.}, 55:545--551, 1949.

\bibitem[OS92]{osgood-stowe1992}
Brad Osgood and Dennis Stowe, \emph{The {S}chwarzian derivative and conformal
  mapping of {R}iemannian manifolds}, Duke Math. J. \textbf{67} (1992), no.~1,
  57--99.

\bibitem[{Sch}17]{schlenker2017}
Jean-Marc {Schlenker}, \emph{Notes on the {S}chwarzian tensor and measured
  foliations at infinity of quasifuchsian manifolds}, Preprint
  \href{https://arxiv.org/abs/1708.01852}{arXiv:1708.01852}, 2017.

\bibitem[Tro92]{tromba1992}
Anthony~J. Tromba, \emph{Teichm{\"u}ller theory in {R}iemannian geometry},
  Lectures in Mathematics ETH Z{\"u}rich, Birkh{\"a}user Verlag, Basel, 1992,
  Lecture notes prepared by Jochen Denzler.

\end{thebibliography}
\end{document}